\documentclass [12pt]{amsart}
\usepackage{hyperref}
\usepackage{cite}
\usepackage{amssymb,bm}
\usepackage{amsmath}
\usepackage{amsthm}
\usepackage{amsfonts,dsfont}
\usepackage{array}
\usepackage{arydshln}
\usepackage{bbm}
\usepackage{bm}
\usepackage[makeroom]{cancel}
\usepackage{graphicx}
\usepackage[rgb,usenames,dvipsnames,svgnames,table]{xcolor}
\usepackage[toc,page]{appendix}
\usepackage{layout}
\usepackage[normalem]{ulem}
\usepackage{mathtools}
\usepackage{multicol}
\usepackage{mathrsfs}
\usepackage{comment}
\usepackage{enumerate,enumitem}
\usepackage{float}
\usepackage{circuitikz}

\newtheorem{theorem}{Theorem}[section]

\newtheorem{definition}[theorem]{Definition}
\newtheorem{proposition}[theorem]{Proposition}

\newtheorem{corollary}[theorem]{Corollary}
\newtheorem{lemma}[theorem]{Lemma}
\newtheorem{remark}[theorem]{Remark}
\newtheorem{example}[theorem]{Example}
\newtheorem{examples}[theorem]{Examples}
\newtheorem{foo}[theorem]{Remarks}

\newtheorem{open question}[theorem]{Open Question}
\newtheorem{c/p}[theorem]{Conjecture/Proposition}

\newcommand{\abs}[1]{\left|#1\right|} % absolute value

\newcommand{\xt}{X_{t}}
\newcommand{\yt}{Y_{t}}
\newcommand{\bbr}{\mathbb{R}}
\newcommand{\calI}{\mathcal I}
\def\vint{\mathop{\mathchoice%
 {\setbox0\hbox{$\displaystyle\intop$}\kern 0.22\wd0%
 \vcenter{\hrule width 0.9\wd0}\kern -0.9\wd0}%
 {\setbox0\hbox{$\textstyle\intop$}\kern 0.2\wd0%
 \vcenter{\hrule width 0.9\wd0}\kern -0.9\wd0}%
 {\setbox0\hbox{$\scriptstyle\intop$}\kern 0.2\wd0%
 \vcenter{\hrule width 0.9\wd0}\kern -0.9\wd0}%
 {\setbox0\hbox{$\scriptscriptstyle\intop$}\kern 0.2\wd0%
 \vcenter{\hrule width 0.9\wd0}\kern -0.9\wd0}}%
 \mathopen{}\int}

\newcommand{\R}{\mathbb R}

\newcommand{\M}{\mathbb M}

\newcommand{\be}{\beta}

\newcommand{\Ric}{\mathfrak Ric}

\DeclareMathOperator{\Cut}{Cut}

\title[Stochastic Schwarz lemma on K\"ahler manifolds]{The stochastic Schwarz lemma on K\"ahler manifolds by couplings and its applications}

\author[Myeongju Chae]{Myeongju Chae{$^{\ddagger}$}}
\thanks{\footnotemark {$\ddagger$} Research was supported in part by  NRF-2018R1A1A3A04079376.}
\address{School of applied mathematics and computer engineering\\
 Hankyong National University, Anseong\\
  17579 Republic of Korea}
\email{mchae@hknu.ac.kr}

\author[Gunhee Cho]{Gunhee Cho{$^{\dagger\dagger}$}}
\thanks{\footnotemark {$\dagger\dagger$} Research was supported in part by the Simons Travel Grant.}
\address{Department of Mathematics\\
	University of California, Santa Barbara\\
	552 University Rd, Isla Vista, CA 93117.}
\email{gunhee.cho@math.ucsb.edu}

\author[Maria Gordina]{Maria Gordina{$^{\dag}$}}
\thanks{\footnotemark {$\dag$} Research was supported in part by NSF Grant DMS-1954264.}
\address{$^{\dag}$ Department of Mathematics\\
University of Connecticut\\
Storrs, CT 06269,  U.S.A.}
\email{maria.gordina@uconn.edu}

\author{Guang Yang}
\address{Department of Mathematics\\
Purdue University\\
West Lafayette, IN 47907, U.S.A.
}
\email{yang2220@purdue.edu}



\date{\today \ \emph{File:\jobname{.tex}}}

\begin{document}

\begin{abstract}
We first provide a stochastic formula for the Carath\'eodory distance in terms of general Markovian couplings and prove a comparison result between the Carath\'eodory distance and the complete K\"ahler metric with a negative lower curvature bound using the Kendall-Cranston coupling. This probabilistic approach gives a version of the Schwarz lemma on complete non-compact K\"ahler manifolds with a further decomposition Ricci curvature into the orthogonal Ricci curvature and the holomorphic sectional curvature, which cannot be obtained by using Yau--Royden's Schwarz lemma. We also prove coupling estimates on quaternionic K\"ahler manifolds. As a byproduct, we obtain an improved gradient estimate of positive harmonic functions on K\"ahler manifolds and quaternionic K\"ahler manifolds under lower curvature bounds.
\end{abstract}

\maketitle

\tableofcontents

\section{Introduction}

The \emph{Carath\'eodory pseudo--distance} $c_{\M}$ on a complex manifold $\M$ is defined as

\begin{equation}\label{eq:Car}
c_{\M}(x,y):=\sup_{f\in \operatorname{Hol}(\M, \mathbb{D})} \rho_{\mathbb{D}}(f(x), f(y)).
\end{equation}
Here $\operatorname{Hol}(\M, \mathbb{D})$ is the collection of all holomorphic functions from $\M$ to $\mathbb{D}$ and we denote the Poincar\'e distance on a unit disk $\mathbb{D}$ in $\mathbb{C}^1$ by \[
\rho_{\mathbb{D}} (z, z^{\prime}):= \left| \frac{z-z^{\prime}}{1-\overline{z^{\prime}}z}  \right|.
\]
Note that $c_{\M}$ is called a pseudo--distance because there could be $x\neq y, x, y\in \M$ with $c_{\M}(x,y)=0$. The triangle inequality, however, is always satisfied. In the rest of this paper, we will simply refer to $c_{\M}$ as \emph{Carath\'eodory distance} when there is no risk of confusion.

Little is known about properties of holomorphic functions on non-compact complex manifolds except for special bounded domains in $\mathbb{C}^n$. For instance, even existence of bounded holomorphic functions on such manifolds is an open question. Existence of non-constant holomorphic functions can be shown on non-compact K\"ahler manifolds by solving the $\overline{\partial}$-equation with H\"ormander's $L^2$-estimate, but it is difficult to deal with boundedness as in \cite{Hormander1965a}. In this regard, the Carath\'eodory distance plays a prominent role, as it provides quantitative information about non-constant bounded holomorphic functions. Moreover, among the invariant distances that satisfy the distance decreasing property, the Carath\'eodory distance is the smallest invariant distance (see, for example, \cite{JarnickiPflugBook2013}).

A fundamental tool for studying the metric and distance of negatively curved complex manifolds is the Yau-Royden's Schwarz lemma \cite{YauST1978b, Royden1980}. However, although Yau-Royden's Schwarz lemma implies that on a complex manifold the Carath\'eodory--Reiffen metric $\gamma_{\M}$ is bounded above by a complete K\"ahler metric $\omega$ with a negative Ricci curvature negative lower bound, and therefore the Carath\'eodory distance is bounded by the geodesic distance $d_\omega$ of $\omega$, such a comparison is not sharp in general. For example, non-sharpness of such a comparison between $c_\M$ and $d_\omega$ follows from the fact that the Carath\'eodory distance $c_\M$ could be different from the inner-Carath\'eodory pseudo-distance (see Section 2.2). On the other hand, it is known that sharp Laplacian and volume comparisons with complex space forms usually require a decomposition of the  Ricci curvature into the orthogonal Ricci curvature and the holomorphic sectional curvature  as in \cite{NiZheng2018}, whereas {Yau--Royden's Schwarz lemma requires the Chern-Lu formula \cite{YauST1978b}, which is not easily adjusted to the further decomposition of the Ricci curvature.

In this paper, we derive a new stochastic formula of the Carath\'eodory distance on complete K\"ahler manifolds using couplings of Brownian motions. A version of the Schwarz lemma on K\"ahler manifolds is provided, which can be adjusted under the change of the Ricci curvature in the way described earlier. To our best knowledge, this is the first case of probabilistic techniques being used in the study of complex geometry.

We briefly explain some basic definitions of couplings techniques in order to state our main results. More details can be found in Section~\ref{s.CouplingsPrelim}. For a diffusion $W_{t}$ on $\M$, we say $\left( X_{t}, Y_{t}\right)$ on $\M\times \M$ is a \emph{coupling} of $W_{t}$ if the marginal processes $X_{t}$ and $Y_{t}$ have the same distribution as $W_{t}$.  A coupling is called Markovian if $\{(X_{t}, Y_{t}) \}$ is a Markov process with respect to its natural filtration $\mathcal{F}_{t}=\sigma\{(X_s, Y_s), s\leqslant t\}$. We define the \emph{coupling time}
\[
\tau\left( X, Y \right):=\inf \left\{ t > 0 : X_{t}=Y_{t} \right\}.
\]
It is always assumed that $X_{t}=Y_{t}$ for all $t >\tau\left( X, Y \right)$.

We are now ready to state our main results.
\begin{theorem}\label{Caratheodory}
On a complete K\"ahler manifold $\M$, for any $x,y\in \M$ with $x \not= y$, we consider a Markovian coupling $\left( X_{t}, Y_{t} \right)$ of Brownian motions with the coupling time $\tau\left( X, Y \right)$. Then, for all $t\geqslant 0$ we have
\begin{equation}\label{eq:1}
c_{\M}(x,y)=\sup_{f\in \operatorname{Hol}\left( M, \mathbb{D} \right), f(y)=0 }\left|\mathbb{E}^{x,y} \left[f(X_{t})-f(Y_{t}) \right]\cdot \mathbbm{1}(\tau\left(X, Y \right)>t) \right|,
\end{equation}
where $\mathbbm{1}(\tau\left(X, Y \right)>t)$ is the indicator function of the event $\{ \tau\left(X, Y \right)>t\}$.
In particular,
\begin{equation}\label{eq:2}
c_{\M}(x,y)\leqslant 2\mathbb{P}^{ x, y }(\tau\left(X, Y \right)>t),
\end{equation}
and by letting $t \rightarrow \infty$,
\begin{equation}\label{eq:3}
c_{\M}(x,y)\leqslant 2\mathbb{P}^{ x, y }(\tau\left(X,  Y \right)=\infty).
\end{equation}
\end{theorem}

\begin{remark}
Historically, the study of Carath\'eodory--Reiffen metric and the Carath\'eodory distance are split into upper and lower bound estimates on mostly special classes of pseudo-convex domains in $\mathbb{C}^n$. For instance, the upper bound was studied for various settings in \cite{LookKH1958, HahnKyongT1976, HahnKyongT1977, AhnGaussierKim2016}. The lower bound of integrated Carath\'eodory-Reiffen metric on complete simply connected non-compact K\"ahler manifolds was recently established by one of the authors in \cite{GHKH21}.  On the other hand, our result \eqref{eq:1} gives an equality for Carath\'eodory distance, which provides more precise information.
\end{remark}

When $\M$ is a unit disk $\mathbb{D}$ in $\mathbb{C}^1$, Theorem~\ref{Caratheodory} leads to the following interesting version of a stochastic Schwarz lemma.

\begin{corollary}[Stochastic Schwarz lemma on a disk]\label{Stochastic-Schwarz}
On a unit disk $\mathbb{D}$ in $\mathbb{C}^1$, for any $x \in \mathbb{D}\backslash\left\{0\right\}$  with $y=0$ and any $t\geqslant 0$,  consider a Markovian coupling $\left( X_{t}, Y_{t} \right)$ of Brownian motions with the coupling time $\tau\left( X, Y \right)$. Then for any holomorphic function $f : \mathbb{D} \rightarrow \mathbb{D}$ with $f(0)=0$
\begin{equation}\label{eq:SS1}
		\left|\mathbb{E}^{x,y} \left[f(X_{t})-f(Y_{t}) \right]\cdot\mathbbm{1}(\tau\left(X, Y \right)>t)\right|\leqslant  \left|\mathbb{E}^{x,y} \left[X_{t}- Y_{t} \right]\cdot \mathbbm{1}(\tau\left(X, Y \right)>t)\right|.
\end{equation}
If the equality holds for some $x, y$, then $f$ must be a rotation. 	
\end{corollary}
For the next theorem, we use the Lindvall-Rogers (Kendall-Cranston) coupling. In particular, Theorem~\ref{main kahler} improves the estimate in \cite{Cranston1991} because of the stronger curvature assumptions as we point out in Remark \refeq{remark}.  See Section \ref{Kahler manifolds} for definitions of $H$ and $Ric^{\perp}$, and Section 6.1 for
$\tau_{(B(x_0, 2\delta)}(X)$.

\begin{theorem}\label{main kahler}
Let $(\M, g)$ be a complete non-compact K\"ahler manifold of the complex dimension $n$. Assume that $H \geqslant 4k_1$ and  $\operatorname{Ric}^\perp \geqslant (2n-2)k_2$ for some $k_1, k_2 \in \mathbb{R}$ with $k_1, k_2< 0$. Let $\left\{ X_{t} \right\}$ be the Markov process with the infinitesimal generator $L=\frac{1}{2} \Delta_{g}+Z$, where $Z$ is a smooth vector field for which there is a constant $m$ such that
\begin{align}\label{zsize}
|Z(x)| \leqslant m, x\in \M.
\end{align}
\begin{comment}
Suppose $\left\{ Y_{t} \right\}$ is the reflected process from $(X_{t})$.
\end{comment}
Let $\left( X, Y \right)$ be the Lindvall-Rogers coupling. Then there is a constant $c=c(k_1, k_2,n)$ such that for all $x, y\in B(x_0,\delta)$
\begin{align}\label{eq:0.2}
		\mathbb{P}^{ x, y }(\tau\left(X, Y \right)&> \tau_{B(x_0,2\delta)}(X)\wedge \tau_{B(x_0,2\delta)}(Y))
\\
&\leqslant c(\frac{1}{\delta}+1)\rho_{\M}(x,y), \notag
\end{align}
where $\rho_{\M}$ is the geodesic distance of $g$ on $\M$. Furthermore,
\begin{equation}\label{eq:0.3}
\mathbb{P}^{ x, y }(\tau\left(X, Y \right)=\infty)\leqslant \left(8((n-1)\sqrt{|k_2|}  + \sqrt{|k_1|})+2m \right)\rho_{\M}(x,y).
\end{equation}
\end{theorem}

By combining Theorems~\ref{Caratheodory} and ~\ref{main kahler}, we have the comparison between the Carath\'eodory distance and the geodesic distance on a complete K\"ahler manifold with a negative curvature lower bound.

\begin{corollary}[Stochastic Schwarz lemma on K\"ahler manifolds]
	Let $(\M,g)$ be a complete non-compact K\"ahler manifold of the complex dimension $n$. Under the settings of Theorem~\ref{main kahler} with $m\equiv 0$, we have for any $x,y \in \M$
	\begin{equation}\label{S-Schwarz lemma}
		c_{\M}(x,y) \leqslant 4\left( 4((n-1)\sqrt{|k_2|}  + \sqrt{|k_1|}) \right)\rho_{\M}(x,y).
	\end{equation}
\end{corollary}
\begin{remark}
It is worth mentioning that, by Yau-Royden's lemma  \cite{YauST1978b,Royden1980} one can give a comparison between the Carath\'eodory distance and the complete K\"ahler metric with negative Ricci curvature lower bound. Our previous formula with a coupling time estimate can be applied under the further decomposition of the Ricci tensor. Thus, \eqref{S-Schwarz lemma} can be viewed as a Schwarz lemma on complete K\"ahler manifolds based on coupling methods. However, it is restricted to the comparison with the Carath\'eodory distance only.
\end{remark}

In addition, as a by-product of the coupling technique we prove a gradient estimate for real-valued harmonic functions. 

%For the local version of gradient estimates on complete noncompact Riemannian manifolds, as Cranston observed in \cite{Cranston1991}, there is a constant $c=c(k_1,k_2,n,m)$ such that whenever $\delta>0$ and $Lu=0$ in some $B(x_0,2\delta)$ then 
%\begin{equation}\label{eq:local1}
%	|\nabla u(x)| \leqslant c(\frac{1}{\delta}+1)\sup_{B(x_0,3\delta/2)} u,
%\end{equation}
%for $x\in B(x_0,\delta)$, and under the further assumption that negative Riemannian sectional curvature upper bound, ~\eqref{eq:local1} is known to be improved as
%\begin{equation}
%	|\nabla u(x)| \leqslant \frac{c}{\delta} u(x), x\in B(x_0,\delta)
%\end{equation}
%Here $c>0$ depends on the sectional curvature upper bound and the constant from ~\eqref{eq:local1}.  

We emphasize that due to the Ricci curvature decomposition, our result is stronger than the same type of estimate on Riemannian manifolds \cite{MR3215340}.

\begin{corollary}[Gradient estimates for harmonic functions on K\"ahler manifolds]\label{harmonic}
~Suppose the manifold $(\M, g)$ is as in Theorem~\ref{main kahler}. If $Lu$=0 on $\M$ and $u$ is bounded and positive, then
\begin{equation}\label{eq:harnack}
|\nabla u(x)| \leqslant \left(8((n-1)\sqrt{|k_2|}  + \sqrt{|k_1|})+2m \right)||u||_{\infty}.
\end{equation}
\end{corollary}

We remark that for the local gradient estimate, $||u||_{\infty}$ in the right-hand side of ~\eqref{eq:harnack} can be replaced by $|u(x)|$ by applying known results of local gradient estimate of positive harmonic function, for example, Cheng-Yau \cite{MR0385749}, and improved local version by O. Munteanu \cite{MR2869128}, and Cranston's coupling approach \cite{Cranston1991} and the different stochastic approach of A. Thalmaie and F-Y-Wang \cite{MR1622800}. 

Along similar lines, we can also provide coupling estimates and their consequences on quaternionic K\"ahler manifolds.

\begin{theorem}\label{quaternion-kahler}
Let $(\M,g)$ be a complete non-compact quaternionic K\"ahler manifold of the complex dimension $n$.
Suppose $k_1,k_2 \in \mathbb{R}$ with $k_1, k_2<0$. Assume that $Q \geqslant
12k_1$ and that $\operatorname{Ric}^\perp \geqslant (4n-4)k_2$. Let $(X_{t})$ be the Markov process  corresponding with the infinitesimal generator $L=\frac{1}{2} \triangle_{g}+Z$, where $Z$ is a smooth vector field for which there is a constant $m$ such that
	\begin{align}\label{zsize}
		|Z(x)| \leqslant m, \quad x\in M,
	\end{align} and let $\left( X, Y \right)$ be the Lindvall-Rogers mirror coupling. Then
	\begin{equation}\label{eq:0.5}
		P^{x,y}(T(X,Y)=\infty)\leqslant \left(8(n-1)\sqrt{|k_2|}  + 24\sqrt{|k_1|}+2m \right)\rho_{\M}(x,y).
	\end{equation}
\end{theorem}

\begin{corollary}[Gradient estimates for harmonic functions on quaternionic K\"ahler manifolds]
Let $(\M, g)$ be as in Theorem~\ref{quaternion-kahler}. If $Lu$=0 on $\M$ and $u$ is bounded and positive, then
	\begin{equation*}
		|\nabla u(x)| \leqslant  \left(8(n-1)\sqrt{|k_2|}  + 24\sqrt{|k_1|}+2m \right)||u||_{\infty}.
	\end{equation*}
\end{corollary}

\section{Geometric preliminaries: K\"ahler and quaternionic K\"ahler manifolds}

In this section, for the sake of completeness, we give the definitions we will be using in this paper. We refer to \cite{BaudoinYang2022}  for more details.
Throughout the paper, let $(\M,g)$ be a smooth complete Riemannian manifold. Denote by $\nabla$ the Levi-Civita connection on $\M$.

\subsection{K\"ahler manifolds} \label{Kahler manifolds}
After \cite{NiZheng2018} and \cite{NiZheng2019}, we will be considering the following type of curvatures on K\"ahler manifolds. Let
\[
R(X,Y,Z,W)=g ( (\nabla_X \nabla_Y -\nabla_Y \nabla_X -\nabla_{[X,Y]} )Z, W)
\]
be the Riemannian curvature tensor on $(\M,g)$.

We have two natural connections on complex manifolds. The Chern connection $\nabla^{c}$  is compatible with the Hermitian metric and the complex structure $J$, and the Levi-Civita connection $\nabla$  is a torsion-free connection compatible with the induced Riemannian metric. The two connections coincide precisely when a hermitian metric on a complex manifold is K\"ahler.

\begin{definition}[K\"ahler manifold]
	The manifold $(\M,g)$ is called a K\"ahler manifold if there exists a smooth  $(1,1)$ tensor $J$ on $\M$ that satisfies the following properties.
	\begin{itemize}
		\item For every $x \in \M$, and $X,Y \in T_x\M$, $g_x(J_x X,Y)=-g_x(X,J_xY)$;
		\item For every $x \in \M$, $J_x^2=-\mathbf{Id}_{T_x\M} $;
		\item $\nabla J$=0.
	\end{itemize}
The map $J$ is called a complex structure.
\end{definition}

We decompose the complexified tangent bundle $T_{\mathbb{R}}\M\otimes_{\mathbb{R}}\mathbb{C}=T'\M\oplus \overline{T'\M}$, where $T'\M$ is the eigenspace of $J$ with respect to the eigenvalue $\sqrt{-1}$ and $\overline{T'\M}$ is the eigenspace of $J$ with respect to the eigenvalue $-\sqrt{-1}$. We can identify $v, w$ as real tangent vectors, and $\eta,\xi$ as corresponding holomorphic $(1,0)$ tangent vectors under the $\mathbb{R}$-linear isomorphism $T_{\mathbb{R}}\M \rightarrow T'\M$, i.e. $\eta=\frac{1}{{\sqrt{2}}}(v-\sqrt{-1}Jv), \xi=\frac{1}{{{\sqrt{2}}}}(w-\sqrt{-1}Jw)$. The components of the curvature $4$-tensor of the Chern connection associated with the Hermitian metric $g$ of a (complex) $n$-dimensional complex manifold are given by
\begin{align*}
	& R_{i\overline{j}k\overline{l}}:= R(\frac{\partial}{\partial z_i},\frac{\partial}{\partial z_i},\frac{\partial}{\partial z_i},\frac{\partial}{\partial z_i})
	\\
	& =g \left( \nabla^c_{\frac{\partial}{\partial z_i} } \nabla^c_{\frac{\partial}{\partial \overline{z_j}} }\frac{\partial}{\partial z_k} -\nabla^c_{\frac{\partial}{\partial \overline{z_j}} } \nabla^c_{\frac{\partial}{\partial z_i}}\frac{\partial}{\partial z_k} -\nabla^c_{[{\frac{\partial}{\partial z_i} },{\frac{\partial}{\partial \overline{z_j}} }]} {\frac{\partial}{\partial z_k} }, {\frac{\partial}{\partial \overline{z_l}}}\right)
	\\
	& =-\frac{\partial^2 g_{i\overline{j}}}{\partial z_k \partial \overline{z}_l}+\sum_{p,q=1}^{n}g^{q\overline{p}}\frac{\partial g_{i\overline{p}}}{\partial z_k}\frac{\partial g_{q\overline{j}}}{\partial \overline{z}_l},
\end{align*}
where $i,j,k,l\in \left\{1.\cdots, n \right\}$ and $\frac{\partial}{\partial z_1},\cdots,\frac{\partial}{\partial z_n}$ is a standard (local) basis for $T' \M$.

The holomorphic sectional curvature of the K\"ahler manifold $(\M,g,J)$ is defined as
\[
H(X)=\frac{R(X,JX,JX,X)}{g(X,X)^2}.
\]

The orthogonal Ricci curvature  of the K\"ahler manifold $(\M,g,J)$ is defined for a vector field  $X$ such that $g(X,X)=1$ by
\[
\operatorname{Ric}^\perp (X,X)=\operatorname{Ric} (X,X)-H(X),
\]
where $\operatorname{Ric}$ is the usual Riemannian Ricci tensor of $(\M,g)$. The table below shows the curvature of the K\"ahler model spaces $\mathbb{C}^m$,  $\mathbb{C}P^m$ and $\mathbb{C}H^m$, see \cite{BaudoinYang2022}.

\begin{table}[H]
	\centering
	\scalebox{0.8}{
		\begin{tabular}{|p{1.5cm}||p{1.5cm}|c|>{\centering\arraybackslash}p{1.8cm}|c|  }
			\hline
			$\M$ &  $H$ & $\operatorname{Ric}^\perp$  \\
			\hline
			\hline
			$\mathbb{C}^m$ &  $0 $ & $0$ \\
			$\mathbb{C}P^m$ & 4 & $2m-2$  \\
			$\mathbb{C}H^m$  &  -4 & $-(2m-2)$  \\
			\hline
	\end{tabular}}
	\caption{Curvatures of K\"ahler model spaces.}
	\label{Table 1}
\end{table}

Unlike the Ricci tensor, $\operatorname{Ric}^{\perp}$ does not admit  polarization, so we never consider $\operatorname{Ric}^\perp (u, v)$ for $u\not= v$. For a real vector field $v$, we can write
\begin{equation*}
	{Ric}^\perp(v,v)=\sum R(v,E_i,E_i,v),
\end{equation*}
where $\left\{e_i \right\}$ is any orthonormal frame of $\left\{v,Jv \right\}^{\perp}$. We will assign index $1,2$ to $v$ and $Jv$ in this summation expression for complex $n$ dimensional K\"ahler manifold $M^n$, and use indices from $3$ to $2n$ for orthonormal frames $\left\{E_i \right\}$ of $\left\{v,Jv \right\}^{\perp}$. Denote by $F_i=\frac{1}{\sqrt{2}}(E_i-\sqrt{-1}J(E_i))$  a unitary frame such that $E_1=v/|v|=:\widetilde{v}$ by following the convention $E_{n+i}=J(E_i)$, then
\begin{align*}
	\frac{1}{|v|^2}\operatorname{Ric}^\perp(v,v)&=\operatorname{Ric}^\perp(\widetilde{v},\widetilde{v})=\operatorname{Ric} (\widetilde{v},\widetilde{v})-R(\widetilde{v},J\widetilde{v},\widetilde{v},J\widetilde{v})\\
	&=\operatorname{Ric}(F_1,\overline{F_1})-R(F_1,\overline{F_1},F_1,\overline{F_1})=\sum_{j=2}^{n}R(F_1,\overline{F_1},F_j,\overline{F_j}).
\end{align*}
In particular, we have $\operatorname{Ric}(F_i,\overline{F_i})=\operatorname{Ric}(E_i,E_i)$, $\operatorname{Ric}^\perp(\widetilde{v},\widetilde{v})=\operatorname{Ric}(F_1,\overline{F_1})-R_{1\overline{1}1\overline{1}}$.

\begin{comment}
Implications from an orthogonal Ricci curvature conditions are quite interesting, especially because of its various geometric implications in complex geometry. For example, A complete K\"ahler manifold $M^n, n\geqslant 2$ with a positive lower bound on the orthogonal Ricci curvature must be compact and always projective \cite[Theorem 1.7]{NiZheng2018}.  Moreover, for $n=2$, a compact $M^2$ which admits a K\"ahler metric with  $Ric^{\perp}>0$ must be biholomorphic to the two-dimensional complex projective space $\mathbb{P}_{\mathbb{C}}^2$, and for $n=3$, a compact K\"ahler manifold under $Ric^{\perp}>0$ must be biholomorphic to either $\mathbb{P}_{\mathbb{C}}^{3}$ or the smooth quadratic hypersurface in  $\mathbb{P}_{\mathbb{C}}^{4}$ as pointed out by \cite[Theorem 1.8]{LeiQingsongFangyang2018}.
\end{comment}

\subsection{Invariant metrics}
Given any complex manifold $\M$, for each $p \in \M$ and a tangent vector $v$ at $p$, define the Carath\'eodory--Reiffen metric by
\[\gamma_{\M}(p;v):=\sup\left\{|df(p)(v)|; f : \M \rightarrow \mathbb{D},  f(p)=0, \text{$f$ holomorphic} \right\}. \]

The inner-Carath\'eodory pseudo-distance on $\M$ is defined by
\[c^{i}_{\M}(x,y):=\inf\{l^{c}(\sigma)(x,y) \} , \]
where the infimums are taken over all piece-wise $C^1$ curves in $\M$ joining $x$ and $y$ and the inner-Carath\'eodory length of a piecewise $C^1$ curve $\sigma : [0,1] \rightarrow \M$ are given by
\[l^{c}(\sigma):=\int_{0}^{1}\gamma_{\M}(\sigma,\sigma'),  \]

The following relation is true in general:
\begin{equation*}
	0\leq c_{{\M}}\leq c^{i}_{\M}.
\end{equation*}

Moreover, $c_{{\M}}$ and $c^{i}_{\M}$ can be different in general (For one such a class of pseudo-convex domains, \cite{Vigue1983a}). In such a class of K\"ahler manifolds, $c_{{\M}}$ and the geodesic distance of the K\"ahler-Einstein metric of Ricci curvature $-1$ must be different due to the Yau-Royden's Schwarz lemma (for example, see \cite{GC21}).

A classical Schwarz lemma on unit disk $\mathbb{D}$ in $\mathbb{C}^1$ has a very natural geometric interpretation: if $f : \mathbb{D} \rightarrow \mathbb{D}$ is a holomorphic function, then
\[
\rho_{\mathbb{D}} (f(z), f(z')) \leqslant \rho_{\mathbb{D}} (z, z'),   \text{ for any } z, z' \in \mathbb{D},
\]
which tells us that holomorphic map of the unit disk into itself decreases the Pointcar\'e distance between points. As a consequence, any automorphism must be the isometry with respect to the Poincar\'e metric. This observation from the Schwarz lemma suggests a natural definition of an invariant metric on a complex manifold $\M$; that is, for $F$ be either a Finsler metric or a hermitian metric. We say $F$ is an \emph{invariant metric} if for any (holomorphic) automorphism $f: M \rightarrow M$,
\[
f^{*}F=F.
\]
The Carath\'eodory-Reiffen metric, the Kobayashi-Royden metric, the complete K\"ahler-Einstein metric of the negative scalar curvature, and the Bergman metric are well-known examples of invariant metrics on arbitrary complex manifolds.

For example, the Poincar\'e metric
\begin{equation}\label{Poincare-metric}
\omega_{p}= \frac{\sqrt{-1}}{2}\partial \overline{\partial}\log\left(\frac{n!}{\pi^n}(1-|z|^2)^{-n-1}\right)
\end{equation}
which has the negative constant holomorphic sectional curvature $\frac{-4}{n+1}$ on the $n$-dimensional complex hyperbolic space $H^n(\mathbb{C})=\{z\in \mathbb{C}^n : ||z||<1 \}$ coincides with those four invariant metrics up to a dimensional constant.

Indeed, for $\M=H^n(\mathbb{C})$,
\begin{equation*}
	\gamma_{H^n(\mathbb{C})}(p;v)=\left[\frac{||v||^2}{1-||p||^2}+\frac{|\langle p,v \rangle|^2}{(1-||p||^2)^2}  \right]^\frac{1}{2},
\end{equation*}
here $||.||$ means the complex Euclidean norm in $\mathbb{C}^n$ (see \cite[p43, Cor 2.3.5]{JarnickiPflugBook2013}). Thus $	\gamma_{H^n(\mathbb{C})}(0;v)=||v||$. On the other hand, $(i,\overline{j})$th component of  $\omega_{p}(0)$ is $(n+1)\delta_{i\overline{j}}$, and the automorphism group of $H^n(\mathbb{C})$ acts on $H^n(\mathbb{C})$ transitively as a group of isometries on both $\omega_{p}$ and $\gamma_{H^n(\mathbb{C})}$, consequently two metrics are the same up to the constant $\sqrt{n+1}$. In particular, the Carath\'eodory distance coincides with the Poincar\'e distance on $H^n(\mathbb{C})$ up to a dimensional constant.

\subsection{Quaternionic K\"ahler manifolds}

\begin{definition}
The manifold $(\M,g)$ is called a \emph{quaternionic K\"ahler manifold} if   there exists a covering of $\M$ by open sets $U_i$ and for each $i$,  3 smooth  $(1,1)$ tensors $I,J,K$ on $U_i$ such that
	
\begin{itemize}
		\item For every $x \in U_i$, and $X,Y \in T_x\M$, $g_x(I_x X,Y)=-g_x(X,I_xY)$,  $g_x(J_x X,Y)=-g_x(X,J_xY)$, $g_x(K_x X,Y)=-g_x(X,K_xY)$ ;
		\item For every $x \in U_i$, $I_x^2=J_x^2=K_x^2=I_xJ_xK_x=-\mathbf{Id}_{T_x\M} $;
		\item For every $x \in U_i$, and $X\in T_x\M$  $\nabla_X I, \nabla_X J, \nabla_X K \in \mathbf{span} \{ I,J,K\}$;
		\item For every $x \in U_i \cap U_j$, the vector space of endomorphisms of $T_x\M$ generated by $I_x,J_x,K_x$ is the same for $i$ and $j$.
	\end{itemize}
\end{definition}
On quaternionic K\"ahler manifolds, we will be considering the following curvatures. As above, let
\[
R(X,Y,Z,W)=g ( (\nabla_X \nabla_Y -\nabla_Y \nabla_X -\nabla_{[X,Y]} )Z, W)
\]
be the Riemannian curvature tensor of $(\M,g)$. We define the quaternionic sectional  curvature of the quaternionic K\"ahler manifold $(\M,g,J)$ as
\[
Q(X)=\frac{R(X,IX,IX,X)+R(X,JX,JX,X)+R(X,KX,KX,X)}{g(X,X)^2}.
\]

We define the orthogonal Ricci curvature  of the quaternionic K\"ahler manifold $(\M,g,I,J,K)$  for a vector field  $X$ such that $g(X,X)=1$ by
\[
\operatorname{Ric}^\perp (X,X)=\operatorname{Ric} (X,X)-Q(X),
\]
where $\operatorname{Ric}$ is the usual Riemannian Ricci tensor of $(\M,g)$. The table below shows the curvature of the quaternion-K\"ahler model spaces $\mathbb{H}^m$,  $\mathbb{H}P^m$ and $\mathbb{H}H^m$, see \cite{BaudoinYang2022}.

\begin{table}[H]
	\centering
	\scalebox{0.8}{
		\begin{tabular}{|p{1.5cm}||p{1.5cm}|c|>{\centering\arraybackslash}p{1.8cm}|c|  }
			\hline
			$\M$ &   $Q$ & $\operatorname{Ric}^\perp$  \\
			\hline
			\hline
			$\mathbb{H}^m$ &  $0 $ & $0$ \\
			$\mathbb{H}P^m$ & 12 & $4m-4$  \\
			$\mathbb{H}H^m$  &  -12 & $-(4m-4)$  \\
			\hline
	\end{tabular}}
	\caption{Curvatures of the quaternionic K\"ahler model spaces.}
	\label{Table 2}
\end{table}

\section{Probabilistic preliminaries: couplings}\label{s.CouplingsPrelim}

%\masha{I am not sure if we need to include a bit on natural filtration, being adopted etc. Please check the text below.}
We gather some basic materials of coupling methods in this section. Let $\left\{ W_{t} \right\}_{t \geqslant 0}$ be stochastic  process on a probability space $\left( \Omega, \mathcal{F}, \mathbb{P}\right)$ taking value in a measure space $\left( E, \mathcal{E} \right)$.

\begin{definition}[Coupling of stochastic processes]\label{d.Coupling3} A coupling of $\left\{ W_{t} \right\}_{t \geqslant 0}$ is a $\left( E\times E, \mathcal{E}\times \mathcal{E} \right)$-valued stochastic process $\{(X_{t}, Y_{t})\}_{t\geqslant 0}$ on a probability space $\left( \Omega', \mathcal{F}', \mathbb{P}'\right)$, such that the marginal processes $\{X_{t}\}_{t\geqslant 0}$ and $\{Y_{t}\}_{t\geqslant 0}$ have the same distribution as $\{Z_{t}\}_{t\geqslant 0}$. A coupling is called Markovian if $\{(X_{t}, Y_{t}) \}_{t\geqslant 0}$ is a Markov process with respect to its natural filtration $\mathcal{F}_{t}=\sigma\{(X_s, Y_s), s\leqslant t\}$.	
\end{definition}

We will only consider Markovian couplings in this paper. It is often beneficial to have a coupling with a fixed starting point $(x,y)\in E\times E$. This can be achieved by choosing the underlying probability, which will be denoted as $\mathbb{P}^{x,y}$, so that
\[\mathbb{P}^{x,y}\left( X_{0}=x, Y_{0}=y\right)=1.\]

Another crucial object in coupling method is the next
\begin{definition}
	Consider a Markovian coupling $\{(X_{t}, Y_{t})\}_{t\geqslant 0}$. The coupling time is defined as
	\[
	\tau(X,Y):=\inf \left\{ t > 0 : X_{t}=Y_{t} \right\}.
	\]
\end{definition}
 We assume that $X_{t}$ and $Y_{t}$ will stick together and move as a single process after the coupling time. %\masha{I think that $\tau$ in general is not the first time the two processes meet and therefore it is not a stopping time in general.}

\begin{definition}[Successful coupling]
A Markovian coupling of two diffusion processes $\left\{ X_{t} \right\}_{t\geqslant 0}$ is said to be successful if $\mathbb{P} \left\{ \tau < \infty \right\}=1$.
\end{definition}
In terms of the coupling time having a successful coupling is equivalent to saying that
\[
\mathbb{P}\left\{ \tau \geqslant t \right\} \xrightarrow[t \to \infty]{}0.
\]
\begin{comment}
While we consider Brownian motion with a drift, we also are interested in how couplings of Brownian motions is related to the geometry of the underlying space. The following definition was introduced by W.~Kendall in \cite{Kendall1986a}.

\begin{definition}[Brownian coupling property (BCP)]\label{d.BCP} A complete Riemannian manifold $M$ has the Brownian coupling property (BCP) if for any $x, y \in M$ we can find a complete probability space $\left( \Omega, \mathcal{F}, \mathbb{P} \right)$, and a filtration of $\sigma$-fields $\mathcal{F}_{t}, t \geqslant 0$, and two Brownian motions $X$ and $Y$ on $M$, not necessarily
independent but both adapted to the filtration, such that
\begin{align*}
& X_{0}=x, Y_{0}=y,
\\
& \mathbb{P}\left( X_{t}=Y_{t} \text{ for some } t \geqslant 0 \right)=1.
\end{align*}
\end{definition}
\end{comment}

\section{Proof of Theorem~\ref{Caratheodory}}

\begin{lemma}\label{lemma:Montel}
Given a complex manifold $\M$, for any distinct two points $x,y \in \M$, there exists the extremal map $f\in \operatorname{Hol}(M, \mathbb{D})$ with respect to the Carath\'eodory pseudo--distance, i.e., there exists a holomorphic function $f: \M \rightarrow \mathbb{D}$ satisfying $c_{\M}(x,y)=\rho_{\mathbb{D}}(f(x), f(y))$.
\begin{proof}
Suppose $\{f_i\}_{i\geqslant 1}$ is a sequence of holomorphic functions for which $\rho_D \left( f_i(x), f_i(y) \right)$ converges to $c_{\M}(x,y)$. Since a family of holomorphic functions from $\M$ to $\mathbb{D}$ are all uniformly bounded. By Montel's theorem \cite{MR1504840}, $\{f_i\}_{i\geq 1}$ must be normal. Hence, $\{f_i\}_{i\geq 1}$ converges to a holomorphic function $f$ and $c_{\M}(x,y)=\rho_{\mathbb{D}}(f(x),f(y))$.
\end{proof}

\end{lemma}

\begin{proof}[Proof of Theorem~\ref{Caratheodory}]
By Lemma~\ref{lemma:Montel}, for any two points in a complex manifold $\M$, there exists a holomorphic function $f: \M \rightarrow \mathbb{D}$ satisfying $c_{\M}(x,y)=\rho_{\mathbb{D}}(f(x),f(y))$. By acting an automorphism of $\mathbb{D}$ and the fact that Carath\'eodory distance is invariant under automorphisms, we may assume that $f(y)=0$. From the formula of the Poincar\'e distance $\rho_{\mathbb{D}}(a,b)=|\frac{a-b}{1-\overline{a}b}|$, we deduce $\rho_{\mathbb{D}}(f(x),f(y))=|f(x)|$.

Since $f$ is harmonic, we have for any $t \geqslant 0$ that
\begin{align}\label{S-Caratheodory}
\notag	c_M(x,y)=|f(x)|&=|f(x)-f(y)|=|\mathbb{E}^{x,y} \left[f(X_{t})-f(Y_{t}) \right]   |
\\
	&=|\mathbb{E}^{x,y} \left[f(X_{t})-f(Y_{t}) \right]\cdot \mathbbm{1}(\tau\left(X, Y \right)>t) |
\\ \notag
	&\leqslant 2||f||_{\infty}\mathbb{P}^{ x, y }(\tau\left(X, Y \right)>t)
\\
\notag
	&\leqslant 2\mathbb{P}^{ x, y }(\tau\left(X, Y \right)>t).\notag
\end{align}
Equations \eqref{eq:2} and \eqref{eq:3} follow immediately.

On the other hand, for any holomorphic function $h : M \rightarrow \mathbb{D}$ with $h(y)=0$, we have
\begin{align*}
	c_M(x,y)&\geqslant \rho_{\mathbb{D}}(h(x),h(y))=|h(x)|=|\mathbb{E}^{x,y} \left[h(X_{t})-h(Y_{t}) \right]   |\\
	&=|\mathbb{E}^{x,y} \left[f(X_{t})-f(Y_{t}) \right]\cdot \mathbbm{1}(\tau\left(X, Y \right)>t) | .
\end{align*}
This proves \eqref{eq:1}.

\end{proof}

\begin{proof}[Proof of Corollary~\ref{Stochastic-Schwarz}]
On the $1$-dimensional complex hyperbolic space $\mathbb{D}$, the Carath\'eodory distance coincide with the Poincar\'e distance with Gaussian curvature $-1$. Thus
\begin{equation*}
	c_{\mathbb{D}}(f(x),f(y))=\rho_{\mathbb{D}}(f(x),f(y))=\rho_{\mathbb{D}}(f(x),0)=|f(x)|,
\end{equation*}
then we can use \eqref{S-Caratheodory} to get the left-hand side of \eqref{eq:SS1}. Similarly,
\begin{equation*}
	c_{\mathbb{D}}(x,y)=\rho_{\mathbb{D}}(x,0)=|x|,
\end{equation*}
and apply \eqref{S-Caratheodory} with the identity map on $\mathbb{D}$ to establish the right-hand side of \eqref{eq:SS1}. The rest of conclusion follows from a classical Schwarz lemma on the unit disk.
\end{proof}

\begin{proof}[Proof of Corollary~\ref{harmonic}]
With a positive real valued $u$ with $Lu=0$,  \eqref{S-Caratheodory} can be changed as follows:
\begin{align*}\label{gradient-harmonic}
	\notag|u(x)-u(y)|&=|\mathbb{E}^{x,y} \left[u(X_{t})-u(Y_{t}) \right]   |\\  \notag
	&=|\mathbb{E}^{x,y} \left[u(X_{t})-u(Y_{t}) \right]\cdot \mathbbm{1}(\tau\left(X, Y \right)>t)| \\ \notag
	&\leqslant ||u||_{\infty}\mathbb{P}^{ x, y }(\tau\left(X, Y \right)>t)\\\notag
	&\leqslant \mathbb{P}^{ x, y }(\tau\left(X, Y \right)>t).\notag
\end{align*}	
Now Corollary~\ref{harmonic} follows from Theorem~\ref{main kahler},
\end{proof}

\section{Ingredients for the  Kendall-Cranston coupling}
We gather some basic results that we will need in the next section. Let $(\M,g)$ be a complete Riemannian manifold with real dimension $n$ and denote by $d$ the Riemannian distance on $\M$.  The index form of a vector field $X$ (with not necessarily vanishing endpoints) along a geodesic $\gamma$ is defined by
\begin{align*}
I (\gamma,X,X): = \int_0^T \left( \langle \nabla_{\gamma'}  X, \nabla_{\gamma'} X  \rangle-\langle R(\gamma',X)X,\gamma'\rangle \right) dt,
\end{align*}
where $\nabla$ is the Levi-Civita connection and $R$ is the Riemann curvature tensor of $\M$.

We will denote by
\[
\Cut (\M):=\left\{ (x,y) \in \M \times \M, \, x \notin \Cut (y) \right\}.
\]
Here $\Cut (y)$ means the cut locus at $y$ and for $(x,y) \notin \Cut (\M)$ we denote
\[
\mathcal{I}(x,y)=\sum_{i=1}^{n-1} I( \gamma, Y_i,Y_i)
\]
where $\gamma$ is the unique length parametrized geodesic from $x$ to $y$ and $\{ Y_1, \cdots, Y_{n-1} \}$ are Jacobi fields such that at both $x$ and $y$, $\{ \gamma', Y_1, \cdots, Y_{n-1} \}$ is an orthonormal frame.

Throughout the paper, we consider the comparison function:

\begin{equation}\label{comparison}
G(k,r) = \begin{cases}  -2 \sqrt{ k} \tan \frac{\sqrt{k}r}{2} & \text{if $k > 0$,} \\
0 & \text{if $k = 0$,}\\ 2 \sqrt{|k|} \tanh \frac{\sqrt{|k|}r}{2} & \text{if $k < 0$.} \end{cases}
\end{equation}

\subsection{Index comparison theorems}

Let $(\M,g,J)$ be a complete K\"ahler with complex dimension $m$ (i.e. the real dimension is $2m$). The holomorphic sectional curvature of $\M$ will be denoted by $H$ and the orthogonal Ricci curvature by $\operatorname{Ric}^\perp$
{as introduced in Section 2.1.}

\begin{comment}
As before, we denote by $d(x,y)$ the Riemannian distance between $x,y \in \M$.
\end{comment}

\begin{theorem}\label{comparison index-K}
Let $k_1,k_2 \in \mathbb{R}$. Assume that $H \geqslant 4k_1$ and that $\operatorname{Ric}^\perp \geqslant (2m-2)k_2$. For every $(x,y) \notin \Cut (\M)$, one has
\[
\mathcal{I}(x,y) \leqslant (2m-2)G(k_2,d(x,y)) +2 G(k_1,2d(x,y)).
\]
\end{theorem}

\begin{proof}
\cite[Theorem 2.1]{BaudoinChoYang2022}.
\end{proof}

%We also denote by $d(x,y)$ the Riemannian distance between $x,y \in \M$.

\begin{theorem}\label{comparison index-Q}
Let $k_1,k_2 \in \mathbb{R}$. Assume that $Q \geqslant 12k_1$ and that $\operatorname{Ric}^\perp \geqslant (4m-4)k_2$.  For every $(x,y) \notin \Cut (\M)$,
\[
\mathcal{I}(x,y) \leqslant (4m-4)G(k_2,d(x,y)) +6 G(k_1,2d(x,y)).
\]
\end{theorem}

\begin{proof}
\cite[Theorem 2.3]{BaudoinChoYang2022}.
\end{proof}

\subsection{Laplacian comparison theorems}

We introduce the following function.

\begin{equation}\label{Laplacian-comparison}
F(k,r) =
\begin{cases}
\sqrt{ k} \cot {\sqrt{k}r} & \text{if $k > 0$,} \\
\frac{1}{r} & \text{if $k = 0$,}\\ \sqrt{ |k|} \coth {\sqrt{|k|}r} & \text{if $k < 0$.}
\end{cases}
\end{equation}

The following theorems follows from proofs in \cite{NiZheng2018, BaudoinYang2022} with a slight modification with $k_1,k_2$ (also see \cite{GHMG21}).

\begin{theorem}
Let $k\in \mathbb{R}$. Assume that $H\geqslant 4k_1$ and that $\Ric^{\perp}\geqslant (2m-2)k_2$. Let $x_0 \in \M$ and denote $r(x)=d(x_0,x)$. Then, pointwise outside of the cut-locus of $x_0$, and everywhere in the sense of distribution, one has
\begin{equation}
\triangle r \leqslant (2m-2)F(k_2,r)+2F(k_1,2r).
\end{equation}
\end{theorem}

\begin{theorem}
	Let $k\in \mathbb{R}$. Assume that $Q \geqslant 12k_1$ and that $\operatorname{Ric}^\perp \geqslant (4m-4)k_2$. Let $x_0 \in \M$ and denote $r(x)=d(x_0,x)$. Then, pointwise outside of the cut-locus of $x_0$, and everywhere in the sense of distribution, one has
	\begin{equation}
		\triangle r \leqslant (4m-4)F(k_2,r)+6F(k_1,2r).
	\end{equation}
\end{theorem}

\subsection{The Kendall-Cranston coupling on complete Riemannian manifolds}

%\gunhee{I feel it seems better to put some more explanation on a reflection coupling, if it is possible to address the intuition about reflection coupling to people who only knows K\"ahler geometry.}

Let $(\M,g)$ be a complete Riemannian manifold. We denote by $O(\M)$ the orthonormal frame bundle of $\M$ and $\pi: O(\M)\rightarrow \M$ the projection onto $\M$. The Levi-Civita  connection gives rise to the horizontal lift
\[H: T\M\rightarrow TO(\M). \]
For any $u\in O(\M)$ and $e\in \R^d$, $H_v(u)$ is the horizontal lift of $ue\in T_{\pi u}\M$.

For a bounded smooth vector filed $Z$ on $\M$, we consider the horizontal diffusion given by
\begin{equation*}
	du_{t}=\sqrt{2} \sum_{i=1}^{d}H_{e_i}(u_{t})\circ dB_{t}+ H_Z(u_{t})dt,\ u_0\in O(\M),
\end{equation*}
where $\{e_i\}_{1\leqslant i\leqslant d}$ is the standard basis of $\R^d$ and $B_{t}$ is a $d$-dimensional Brownian motion. The projection process $X_{t}=\pi u_{t}$ is the diffusion on $\M$ generated by $\frac{1}{2}\Delta+Z$.

The Kendall-Cranston coupling is a way to construct Markovian couplings of $X_{t}$ on $\M$. Let $(x,y) \in \Cut(\M)$ that can be connected by a (unique) minimizing geodesic. We define the ``mirror map'' $m_{x,y} : T_x \M \rightarrow T_y \M$ as follows: carry each tangent vector in $T_x \M$ to $T_y \M$ by parallel transport along the unique geodesic between $x$ and $y$, then reflect it in the hyperplane normal to the geodesic in $T_y \M$.

The Kendall-Cranston coupling $(X_{t}, Y_{t})$ of diffusion generated by $\frac{1}{2}\Delta+Z$ with starting points $(x,y)$ is explicitly given by the following system of SDEs,
\begin{equation}\label{eq:coupling_system}
\begin{cases}
d u_{t} = \sqrt{2} \sum_{i=1}^{d}H_{e_i}(u_{t})\circ dB_{t}+ H_Z(u_{t})dt \\
d v_{t} = \sqrt{2} \sum_{i=1}^{d}H_{e_i}(v_{t})\circ dW_{t}+ H_Z(v_{t})dt \\
d W_{t} = ( {v}_{t}^{-1}m_{X_{t},Y_{t}} u_{t} ) dB_{t} \\
X_{t}=\pi u_{t},\ X_0=x  \\
Y_{t}=\pi v_{t},\ Y_0=y \end{cases}
\end{equation}
Note that $W_{t}$ is another $d$-dimensional Brownian motion due to the fact that $m_{X_{t},Y_{t}}$ is an isometry.

With everything prepared, we can now have
\begin{theorem}[Theorem 3 in\cite{Cranston1991}]\label{first-ingredient}
With the same settings as above, let $\rho_{\M}$ be the geodesic distance on $(\M,g)$. Then the following inequality holds.
	\begin{equation}\label{eq:1.7}
	d\rho_{\M}(X_{t},Y_{t})\leqslant 2\be_{t}+\mathcal{I} (X_{t},Y_{t}) dt+\left[\langle Z(Y_{t}),T_{t} \rangle-\langle Z(X_{t}),T_{t} \rangle \right]dt,
	\end{equation}
	where $(\beta_{t})_{t \geqslant 0}$ is a Brownian motion on $\mathbb R$ with its quadratic variation $\langle \beta \rangle_{t}=2t$ and  $T_{t}$ is the tangent vector to $\gamma$ and the integral is along the geodesic $\gamma$.
	
\end{theorem}

\section{Kendall-Cranston coupling}

With  the index form estimates of the previous section in hands, we can use the reflection coupling method by {M.~Cranston},  M.~F.~Chen, and F.~Y.~Wang \cite{Cranston1991, ChenLi1989, WangFY2004b}  (see also \cite[Section 6.7]{HsuEltonBook}) to get coupling time estimates.

\subsection{K\"ahler case}

\begin{proof}[Proof of Theorem~\ref{main kahler}]
Define $\rho_{t}$ by
\[ d\rho_{t} = d \be_{t} + ( 4(n-1)\sqrt{|k_2|}  + 4\sqrt{|k_1|}+2m)dt\]
for the same brownian motion appearing in \eqref{eq:1.7}
with the initial condition
\[ \rho_0 = \rho_{\M}(x, y).\]
We have
\[ \calI(\xt, \yt) \leqslant 4(n-1)\sqrt{|k_2|}  + 4\sqrt{|k_1|}\]
from \eqref{comparison} and Theorem \ref{comparison index-K} and
\[\langle Z(Y_{t}),T_{t} \rangle-\langle Z(X_{t}),T_{t} \rangle \leqslant 2m.\]
Thus by a comparison theorem it follows that
\[ \rho_{\M}(\xt, \yt) \leqslant \rho_{t} \quad \mbox{ for all } t>0 \mbox{ a.s.}\]
(Recall that the integrands in the $dt$ terms are $0$ if $Y$ is inside of the cut-locus of $X_{t}$.)
Now if
\begin{equation*}
\sigma_a(\rho_{\M})\equiv \inf \left\{t>0 : \rho_{\M}(X_{t},Y_{t})=a \right\},
\end{equation*}
then
\begin{equation*}
\tau\left(X, Y \right)=\sigma_0(\rho_{\M})
\end{equation*}
so that
\begin{equation*}
\tau\left(X, Y \right)\leqslant \sigma_0(\rho)
\end{equation*}
and
\begin{equation*}
\sigma_2(\rho_{\M})\geqslant \sigma_2(\rho) \text{ a.s. by the above comparison.}
\end{equation*}
Thus
\begin{equation}\label{eq:2.1}
\mathbb{P}^{ x, y }(\tau\left(X, Y \right)=\infty)\leqslant P^{\rho_0}(\sigma_0(\rho)=\infty).
\end{equation}
Also
\begin{align}\label{eq:2.2}
&\mathbb{P}^{ x, y }(\tau\left(X, Y \right)>\tau_{B(x,\delta)}(X)\wedge \tau_{B(x,\delta)}(Y)\\\notag
&=\mathbb{P}^{ x, y }(\tau\left(X, Y \right)>\tau_{B(x,\delta)}(X)\wedge \tau_{B(x,\delta)}(Y)\wedge \sigma_{2\delta}(\rho_{\M}))\\ \notag
&\leqslant \mathbb{P}^{ x, y }(\tau\left(X, Y \right)>\tau_{B(x,\delta)}(X)\wedge \sigma_{2\delta}(\rho_{\M}))+\mathbb{P}^{ x, y }(\tau\left(X, Y \right)>\tau_{B(x,\delta)}(Y)\wedge \sigma_{2\delta}(\rho_{\M}))\\ \notag
&\leqslant \mathbb{P}^{ x, y }(\sigma_0(\rho)\geqslant \sigma_{\delta}(\rho_{\M}(X,x)\wedge \sigma_{2\delta}(\rho))+\mathbb{P}^{ x, y }(\sigma_0(\rho)\geqslant \sigma_{\delta}(\rho_{\M}(Y,x)\wedge \sigma_{2\delta}(\rho)).\notag
\end{align}
Focusing first on the inequality \eqref{eq:2.1} define
\begin{equation*}
u(\rho_0)=P^{\rho_0}(\sigma_0(\rho)=\infty).
\end{equation*}
We need the following lemma to estimate $u(\rho_0)$.
\begin{lemma}[Lemma 2.1 \cite{WangFY2004b}]
Let $(r_{t})_{t\geqslant 0}$ be the one-dimensional diffusion process generated by $ a\frac{d^2}{dr^2} + b(r) \frac{d}{dr}$, where
$a>0$ is a constant and $b \in C^1 (\bbr)$. Let $r_0>0$ and $\tau_0 : = \{t\geqslant 0: r_{t} =0 \}$. Let
\begin{align*}
\xi(r) : & = \int_0^r exp \left[ -\frac 1a \int_0^s b(t) dt \right] ds, \quad  r\in \bbr, \\
c(u) : & = \frac 1a \sup_{t\in [0, u]} \int_0^t b(s) ds,  \quad  u>0.
\end{align*}
We have
\begin{align}\label{wanglemma}
P (\tau_0 >t) \leqslant \xi(r_0) \inf_{s> r_0} \left\{ \frac{1}{\xi(s)} + \frac{ e^{c(s)}}{\sqrt{a\pi t}}\right\}, \quad t>0.
\end{align}
\end{lemma}
In the context of the paper we apply the lemma to $\rho_{t}$ generated by $\frac 12 \frac{d^2}{d r^2} + b \frac{d}{dr}$ and $\tau_0 = \sigma_0 (\rho)$, where
\[ b = 4(n-1)\sqrt{|k_2|}  + 4\sqrt{|k_1|}+2m.\]
Then $\xi(\rho_0) = \frac ab(1 -e^{-\frac ab \rho_0}) $ and the right hand side of \eqref{wanglemma} is
\[\xi(\rho_0) \frac ba \left (  1 +  \frac{1}{ (a\pi t)^{\frac 14}} + \frac{ (a\pi t)^\frac 14 +1}{ (a\pi t)^{\frac 12}} \right)\]
for sufficiently large $t$. By passing $t \to \infty$, we have
\[ u(\rho_0) \leqslant \frac ba \rho_0 = [8((n-1)\sqrt{|k_2|}  + \sqrt{|k_1|})+2m] \rho_0,\]
which proves \eqref{eq:0.3}.
For \eqref{eq:0.2}, it suffices to handle one term on the extreme right-hand side of \eqref{eq:2.2}, the other being similar. According to Kendall \cite{Kendall1986a}, there is a local time term $L^x$, i.e., an increasing process supported on $C(X)$ such that
\begin{equation}\label{eq:2.3}
d\rho_{M}(X_{t},x)=dw_{t}+(\frac{1}{2}\triangle+Z)\rho_{M}(X_{t},x)dt-dL^{x}_{t},
\end{equation}
where $\left\{w_{t} : t\geqslant 0 \right\}$ is $BM(\mathbb{R}^1)$.
If $X_{t}$ is outside of the cut locus of $x$, by the Laplacian comparison theorem,
\begin{equation*}
\triangle \rho_{M} (X_{t},x)\leqslant (2m-2)F(k_1,\rho_{M}(X_{t},x))+2F(k_2,2\rho_{M}(X_{t},x)).
\end{equation*}
We take $(\frac{1}{2}\triangle+Z)\rho_{M}(X_{t},x)=0$ for $X_{t} \in C(x)$. Thus, using the same $BM(\mathbb{R}^1)$, $w,$ appearing in ~\eqref{eq:2.3} to define $\eta_{t}$ by
\begin{equation*}
d\eta_{t}=dw_{t}+((m-1) \sqrt {|k_1|} \coth \sqrt{|k_1|}\eta_{t} + \sqrt {|k_2|} \coth 2\sqrt{|k_2|}\eta_{t}+ m)dt, \quad  \eta_0=0.
\end{equation*}
By comparison again and \eqref{Laplacian-comparison} one has
\begin{equation*}
\rho_{M}(X_{t},x)\leqslant \eta_{t} \text{ for all } t\geqslant 0 \text{ a.s. }
\end{equation*}
(As before, the $dt$ coefficients for $\rho_{M}(X_{t},x)$ are set equal to $0$ if $X_{t} \in C(x)$.) Thus
\begin{equation}\label{eq:2.4}
\tau_{B(x,\delta)}(X):=\sigma_{\delta}(\rho_{M}(X,x))\geqslant \sigma_{\delta}(\eta) \text{ a.s. }
\end{equation}
Using ~\eqref{eq:2.4} in the first term on the far right-hand side of ~\eqref{eq:2.2}, it appears that
\begin{equation}\label{eq:2.5}
\mathbb{P}^{ x, y }(\sigma_0(\rho)\geqslant \sigma_{\delta}(\rho_{M}(X,x)\wedge \sigma_{2\delta}(\rho)))\leqslant P^{(\rho_0, 0)}(\sigma_0(\rho)> \sigma_{\delta}(\eta)\wedge \sigma_{2\delta}(\rho)).
\end{equation}
The superscript on the probability on the right means $\rho_0$ is the starting point for $\rho,0$ for $\eta.$ Set $B_1=4(m-1)\sqrt{|k_2|}  + 4\sqrt{|k_1|+m} $,  $B_2(\eta)=(m-1)F(k_1,\eta)+F(k_2,2\eta)+m$. From It\^o's formula, we have
\begin{align}\label{eq:2.6} \begin{aligned}
&h(\rho_{t},\eta_{t})=h(\rho_0,0)+\int_{0}^{t}\nabla h(\rho_s,\eta_s)d(2b_s,w_s) \\
&+\int_{0}^{t}(2h_{\rho \rho}+h_{\eta \eta}+B_1h_{\rho}+B_2(\eta)h_\eta)ds+\int_{0}^{t}2h_{\rho \eta}(\rho_s,\eta_s)d \langle b,w \rangle_s
\end{aligned} \end{align}
for a nice function $h,P^{\rho_0,0}$ a.s.
The right-hand side of \eqref{eq:2.5} may be viewed as $E^{(\rho_0,0)}1_{A}(\rho_{t},\eta_{t})$, where $$\tau=\inf \left\{t>0 : \rho_{t} \notin (0,2\delta) \text{ or } \eta_{t}\notin [0,\delta]\right\}$$ and $A=\left\{(2\delta,\eta) : 0\leqslant \eta \leqslant \delta \right\}\cup \left\{(\rho,\delta) : 0<\rho \leqslant 2\delta \right\}$. Thus, $E^{\rho_0,0}1_{A}(\rho_\tau,\eta_\tau)$ may be estimated by selecting a nice $h\geqslant 0$ for which $h|_A\leqslant 1_A$ and evaluating $E^{\rho_0,0}h(\rho_\tau,\eta_\tau)$.

To this end define

\begin{equation*}
h(\rho,\eta) = \begin{cases}  1 & \text{on $A$,} \\
\rho/2\delta & \text{on $\left\{(\rho,\eta) : 0<\rho < 2\delta, 0 < \eta < \delta/2 \right\} $,} \end{cases}
\end{equation*}
and then extend $h$ so that
\begin{equation}\label{eq:2.7}
|h_\eta| \vee |h_\rho| \vee \delta |h_{\rho \rho}| \vee \delta |h_{\eta \eta}| \vee \delta |h_{\rho \delta}| \leqslant \frac{c}{\delta}.
\end{equation}
Note that the choice of $h$ also resolves the singularity of $F(k, \eta)$  at $\eta =0$ by achieving $h_{\eta} =0$ near $\eta =0$.
Thus, using \eqref{eq:2.6} and \eqref{eq:2.7}, it follows that
\begin{align*}
P^{\rho_0,0}(\sigma_0>\sigma_{\delta}(\eta)\wedge \sigma_{2\delta}(\rho))&\leqslant h(\rho_0,0)+\frac{c}{\delta^2}E^{\rho_0,0}\tau\\
&\leqslant \left(\frac{c}{\delta}+c(K,d,m) \right)\rho_0.
\end{align*} For $E^{\rho_0,0}\tau \leqslant c\rho_0 \delta$ we use the reflection property of $\rho$, the Brownian motion with a drift \cite{Kendall1986a}. The proof is complete.

\end{proof}

\begin{remark}\label{remark}
	$H \geqslant 4k_1$ and  $\operatorname{Ric}^\perp \geqslant (2n-2)k_2$ imply the Ricci lower bound $\operatorname{Ric} \geqslant 4k_1+(2n-2)k_2$ and we always have
	\[
	(2n-2)G(k_2,r) +2 G(k_1,2r) \leqslant (2n-1) G \left( \frac{4k_1+(2n-2)k_2}{2n-1},  r \right).
	\]
	by concavity of $k \to G(k,r)$. Combining with Theorem~\ref{first-ingredient}, ~\eqref{comparison} with the same argument in the proof, this reduces to the same coupling result of M. Cranston for Riemannian manifolds with Ricci curvature lower bound \cite{Cranston1991}.
\end{remark}

\subsection{Quarternion K\"ahler case}
The proof works as same except modifying
$\rho_{t}$ by
\[ d\rho_{t} = d \be_{t} + ( 4(n-1)\sqrt{|k_2|}  + 4\sqrt{|k_1|})dt\]

and we apply the lemma to $\rho_{t}$ generated by $\frac 12 \frac{d^2}{d r^2} + b \frac{d}{dr}$ and $\tau_0 = \sigma_0 (\rho)$, where
\[ b = 4(m-1)\sqrt{|k_2|}  + 4\sqrt{|k_1|}.\]
Then $\xi(\rho_0) = \frac ab(1 -e^{-\frac ab \rho_0}) $ and the right hand side of \eqref{wanglemma} is
\[\xi(\rho_0) \frac ba \left (  1 +  \frac{1}{ (a\pi t)^{\frac 14}} + \frac{ (a\pi t)^\frac 14 +1}{ (a\pi t)^{\frac 12}} \right)\]
for sufficiently large $t$. By passing $t \to \infty$, we have
\[ u(\rho_0) \leqslant \frac ba \rho_0 = 8((m-1)\sqrt{|k_2|}  + \sqrt{|k_1|}) \rho_0,\]
which proves \eqref{eq:0.3}.
For \eqref{eq:0.2}, it suffices to handle one term on the extreme right-hand side of \eqref{eq:2.2}, the other being similar. According to Kendall [], there is a local time term $L^x$, i.e., an increasing process supported on $C(X)$ such that
\begin{equation}\label{eq:2.3}
	d\rho_{M}(X_{t},x)=dw_{t}+(\frac{1}{2}\triangle+Z)\rho_{M}(X_{t},x)dt-dL^{x}_{t},
\end{equation}
where $\left\{w_{t} : t\geqslant 0 \right\}$ is $BM(\mathbb{R}^1)$.
If $X_{t}$ is outside of the cut locus of $x$, by the Laplacian comparison theorem,
\begin{equation*}
	\triangle \rho_{M} (X_{t},x)\leqslant (2m-2)F(k_1,\rho_{M}(X_{t},x))+2F(k_2,2\rho_{M}(X_{t},x)).
\end{equation*}
We take $(\frac{1}{2}\triangle+Z)\rho_{M}(X_{t},x)=0$ for $X_{t} \in C(x)$. Thus, using the same $BM(\mathbb{R}^1)$, $w,$ appearing in ~\eqref{eq:2.3} to define $\eta_{t}$ by
\begin{equation*}
	d\eta_{t}=dw_{t}+((m-1) \sqrt {|k_1|} \coth \sqrt{|k_1|}\eta_{t} + \sqrt {|k_2|} \coth 2\sqrt{|k_2|}\eta_{t}+ m)dt, \quad  \eta_0=0.
\end{equation*}
By comparison again and \eqref{Laplacian-comparison} one has
\begin{equation*}
	\rho_{M}(X_{t},x)\leqslant \eta_{t} \text{ for all } t\geqslant 0 \text{ a.s. }
\end{equation*}

\section{Couplings on $H^2(\mathbb{C})$}
In this section, we study Markovian couplings of Brownian motions on $H^2(\mathbb{C})$ as an example. It is known that properties of these couplings are closely related to the harmonic functions. We state the next result and refer to \cite{ChenMFBook2005} page 38 for a proof on Euclidean space; generalization to Riemannian manifolds is straightforward.

\begin{proposition}
	Let $\M$ be a Riemannian manifold. If there exists successful coupling of Brownian motions for any pair of starting points $(x,y)\in \M\times \M$, then any bounded harmonic function on $\M$ must be constant.
\end{proposition}

We know that there are bounded non-constant harmonic functions on $H^2(\mathbb{C})$, so by previous proposition, we cannot have successful coupling of Brownian motions for all different starting points. In fact, we can get a stronger result and show that all Markovian couplings with different starting points are unsuccessful.

\begin{proposition}
	\label{Coupling on H2C}
	Let $(X_{t}, Y_{t})$ be a Markovian coupling of Brownian motions on $H^2(\mathbb{C})$ with starting point $(x,y)$. If $x\neq y$, then the coupling is not successful. In particular,
	\begin{equation*}
		\mathbb{P}(\tau=\infty)\geqslant \frac{\rho^2(x, y)}{4},
	\end{equation*}
	where $\rho(x, y)$ is the Euclidean distance between $x$ and $y$.
\end{proposition}

We prepare some materials regarding Brownian motions on $H^2(\mathbb{C})$. It was proved in \cite{GraczykZak2008} that a Brownian motion on $H^2(\mathbb{C})$ satisfies the following SDE,	
\begin{align*} dX_1(t)&=2\sqrt{1-\abs{X(t)}^2}
\\
& \times \left(\left(1-\frac{\abs{X_1(t)}^2}{1-\sqrt{1-\abs{X(t)}^2}}\right)dB_1(t) -\frac{X_1(t)\overline{X}_2(t)}{1-\sqrt{1-\abs{X(t)}^2}}dB_2(t)  \right),
\\
dX_1(t)&=2\sqrt{1-\abs{X(t)}^2}
\\
& \times \left(  \frac{-\overline{X}_1(t)X_2(t)}{1-\sqrt{1-\abs{X(t)}^2}} dB_1(t)+\left(1-\frac{\abs{X_1(t)}^2}{1-\sqrt{1-\abs{X(t)}^2}}  \right)dB_2(t) \right),
\end{align*}
where $\abs{X(t)}^2=\abs{X_1(t)}^2+\abs{X_2(t)}^2$. By identifying $\mathbb{C}^2$ with $\mathbb{R}^4$ through the canonical isomorphism, $X_{t}$ can be regarded as a diffusion in $\mathbb{R}^4$ given by
\begin{equation}
		\label{H2C BM in R4}
		dX_{t}=A(X_{t})dB_{t},
\end{equation}
where $B_{t}$ is a standard Brownian motion in $\mathbb{R}^4$ and for $x=(a_1,b_1,a_2,b_2)\in \mathbb{R}^4$, and $A(x)$ is given by

\begin{align*}
& A(x)=	2\sqrt{1-\abs{x}^2}
\\
&
\times \begin{pmatrix}
			1-\frac{a^2_1+b^2_1}{1+\sqrt{1-\abs{x}^2}} & 0 & \frac{-(a_1a_2+b_1b_2)}{1+\sqrt{1-\abs{x}^2}}& \frac{b_1a_2-a_1b_2}{1+\sqrt{1-\abs{x}^2}}\\
			0 & 1-\frac{a^2_1+b^2_1}{1+\sqrt{1-\abs{x}^2}} & \frac{a_1b_2-b_1a_2}{1+\sqrt{1-\abs{x}^2}} & \frac{-(a_1a_2+b_1b_2)}{1+\sqrt{1-\abs{x}^2}} \\
			\frac{-(a_1a_2+b_1b_2)}{1+\sqrt{1-\abs{x}^2}} & \frac{a_1b_2-b_1a_2}{1+\sqrt{1-\abs{x}^2}} & 1-\frac{a^2_2+b^2_2}{1+\sqrt{1-\abs{x}^2}} & 0\\
			\frac{b_1a_2-a_1b_2}{1+\sqrt{1-\abs{x}^2}} & \frac{-(a_1a_2+b_1b_2)}{1+\sqrt{1-\abs{x}^2}} & 0 & 1-\frac{a^2_2+b^2_2}{1+\sqrt{1-\abs{x}^2}}
		\end{pmatrix}.
\end{align*}
The symmetric matrix $A(x)$ is positive definite with eigenvalues $2\sqrt{1-\abs{x}^2},\\ 2(1-\abs{x}^2)$, both of multiplicity $2$. But it gets closer to being degenerate as $\abs{x}$ approaches $1$.

If $(X_{t}, Y_{t})$ is a Markovian coupling of Brownian motions, it must admit a generator of the form
\begin{equation*}
	L=\frac{1}{2}\sum_{i,j=1}^{8}a(x,y)_{i,j}\frac{\partial^2}{\partial_i\partial_j},
\end{equation*}
where $a(x,y)\in\mathbb{R}^{8\times 8}$ is a symmetric matrix given by
\begin{equation*}
	\begin{pmatrix}
		A(x)A^T(x) & C(x,y)\\
		C^T(x,y) & A(y)A^T(y)
	\end{pmatrix},
\end{equation*}
for some $C(x,y)\in \mathbb{R}^{4\times 4}$. More precisely, one has
\begin{equation*}
	C(x,y)=A(x)\cdot \frac{d \langle B_{t}, W_{t} \rangle}{dt}\big|_{(x,y)}\cdot A^T(y),
\end{equation*}
where $\langle B_{t}, W_{t} \rangle$ denotes the covariance process of the underlying Brownian motions of $X_{t}$ and $Y_{t}$.  By standard properties of diffusion generators, $a(x,y)$ must be non-negative definite in $H^2(\mathbb{C})\times H^2(\mathbb{C}) \subset \mathbb{R}^8$. As a result
\begin{equation*}
	\hat{A}(x,y)=A(x)A^T(x)+A(y)A^T(y)-2C(x,y)
\end{equation*}
is also non-negative definite. In particular, $Tr\left(\hat{A}(x,y)\right)>0$.

\begin{proof}[ Proof of Proposition~\ref{Coupling on H2C}]
It is readily checked that
\[L\rho^2(X_{t}, Y_{t})=Tr\left(\hat{A}(x,y)\right).    \]
We define a sequence of stopping times
\[\tau_n(X_{t},Y_{t})=\inf\{ t\geqslant 0 : \rho(X_{t},Y_{t})\leqslant \frac{1}{n}   \},\ n\in\mathbb{N}^+.  \]
Obviously $\tau(X_{t},Y_{t})=\lim_{n\rightarrow \infty}\tau_n(X_{t},Y_{t})$. By It\^o's formula
\begin{align*}
	\mathbb{E}\rho^2(X_{\tau_n \wedge t}, Y_{\tau_n \wedge t})=\rho^2(X_0, Y_0)+\mathbb{E}\int_{0}^{\tau_n \wedge t}Tr\left(\hat{A}(X_s,Y_s)\right)ds.
\end{align*}
The above equation can be recast as
\begin{align*}
	\mathbb{E}(\mathbbm{1}(\tau_n>t)\rho^2(X_{t}, Y_{t}) )+\frac{1}{n^2}\mathbb{P}(\tau_n<t)=\rho^2(X_0, Y_0)+\mathbb{E}\int_{0}^{\tau_n \wedge t}Tr\left(\hat{A}(X_s,Y_s)\right)ds.
\end{align*}
Since $X_{t},Y_{t}$ never leave $H^2(\mathbb{C})\subset \mathbb{R}^4$ and $Tr\left(\hat{A}(X_s,Y_s)\right)$ is non-negative, we deduce
\begin{equation*}
	4\mathbb{P}(\tau_n>t)+\frac{1}{n^2}\mathbb{P}(\tau_n<t)\geqslant \rho^2(X_0, Y_0).
\end{equation*}
Sending $n\rightarrow \infty$ gives
\begin{equation*}
	4\mathbb{P}(\tau>t)\geqslant \rho^2(X_0, Y_0) \text{ for all } t \geqslant 0,
\end{equation*}
and we conclude
\[
\mathbb{P}(\tau=\infty)\geqslant \frac{\rho^2(X_0, Y_0)}{4}=\frac{\rho^2(x, y)}{4}>0.
\]
Thus, the coupling $(X_{t}, Y_{t})$ is not successful.
	
\end{proof}

\bibliographystyle{amsplain}
\bibliography{references}

\end{document}